\documentclass[12pt]{amsart}
\usepackage{amsfonts}
\usepackage{latexsym,amscd,amssymb}
\usepackage[all]{xy}
\usepackage{array}
\usepackage{graphicx}
\usepackage{epsfig}
\usepackage{subfigure}

\newtheorem{theorem}{Theorem}[section]
\newtheorem{proposition}[theorem]{Proposition}
\newtheorem{lemma}[theorem]{Lemma}
\newtheorem{corollary}[theorem]{Corollary}
\newtheorem{definition}[theorem]{Definition}
\newtheorem{conjecture}[theorem]{Conjecture}
\newtheorem{example}[theorem]{Example}

\numberwithin{equation}{section}

\begin{document}

\title{Cotangent bundles of toric varieties and coverings of toric hyperk{\"a}hler manifolds}

\author{Craig van Coevering \and Wei Zhang}

\thanks{Mathematics Classification Primary(2000): Primary 53C26, Secondary 53D20, 14L24.\\
\indent The second author is supported by Tian Yuan math Fund. and the Fundamental Research Funds for the Central Universities}

\begin{abstract}

Toric hyperk{\"a}hler manifolds are quaternion analog of toric
varieties. Bielawski pointed out that they can be glued by cotangent
bundles of toric varieties. Following his idea, viewing both toric
varieties and toric hyperk{\"a}her manifolds as GIT quotients, we
first establish geometrical criteria for the semi-stable points.
Then based on these criteria, we show that the cotangent bundles of
compact toric varieties in the core of toric hyperk{\"a}hler
manifold are sufficient to glue the desired toric hyperk{\"a}hler
manifold.

\end{abstract}

\maketitle

\section{Introduction}

Toric varieties are originally defined by the combinatorial data of
the fans(cf. \cite{Fu93} or \cite{Od88}), also studied by Delzant
and Guillemin (\cite{De88}, \cite{Gu94b}) from symplectic quotient
perspective, which has natural connection with the Geometric
Invariant Theory(cf. \cite{MFK94}). Toric hyperk{\"a}hler manifold
is a quaternion analogue of toric variety which carries
hyperk{\"a}hler metric automatically. In \cite{BD00}, Bielawski and
Dancer studied their basic properties: moment map, core, cohomology,
etc. In \cite{Bi99}, Bielawski pointed out that toric
hyperk{\"a}hler manifold can be viewed as gluing the cotangent
bundles of toric varieties, however, he didn't give the explicit way
of gluing.

Following his instruction, we pursue a ``canonical" set of toric
varieties whose cotangent bundles are enough to glue the toric
hyperk{\"a}hler manifold. For this purpose, we combine the
symplectic and GIT quotients methods. In section 2, we give the
definitions and basic properties of toric varieties and toric
hyperk{\"a}hler manifolds as symplectic quotients.

The essential new ingredient of this paper is in section 3. In
\cite{Ko03}, Konno present the GIT quotient construction of toric
hyperk{\"a}hler manifold and give a numerical criterion for the
semi-stable points. Basing on his method, we derive a similar
criterion for the toric varieties. Then we establish a geometric
interpretations of these two semi-stability criteria, which are the
affine analogs of the state sets in projective case due to Dolgachev
and Hu in \cite{DH98}. Namely, for a toric variety $X(\alpha)$ with
hyperplanes arrangement $\mathcal{A}=\{(H_i,u_i)\}_{i=1}^d$, where
$H_i=\{x \in \mathfrak{n}^*|\langle u_i,x\rangle+\lambda_i=0\}$, we
can set the half space

\begin{equation*}
H^{\geq 0}_i=\{x \in \mathfrak{n}^*|\langle u_i,x \rangle+\lambda_i
\geq 0\},
\end{equation*}
and for every $z \in \mathbb{C}^d$ define
$St_{\mathcal{A}}(z)=\bigcap_{i=1}^d F_z(i)$, where

\begin{equation*}F_z(i)=
\begin{cases}
H_i^{\geq 0} & if z_i \neq 0\\
H_i & if z_i = 0
\end{cases}.
\end{equation*}
Thus viewing toric variety $X(\alpha)$ as a GIT quotient, we have

\begin{proposition}\label{pro:tor:sta}
A point $z \in \mathbb{C}^d$ is $\alpha$-semi-stable if and only if
the set $St_{\mathcal{A}}(z) \subset \mathfrak{n}^*$ is nonempty.
\end{proposition}

Similarly, for a toric hyperk{\"a}hler variety $Y(\alpha,
\beta)$(see the detailed definition $St_{\mathcal{A}}(z,w)$ in
section 3), we have
\begin{proposition}\label{pro:hyp:sta}
A point $(z,w) \in \mu_{\mathbb{C}}^{-1}(\beta)$ is
$\alpha$-semi-stable if and only if $St_{\mathcal{A}}(z,w) \subset
\mathfrak{n}^*$ is nonempty.
\end{proposition}

Above two criteria can be viewed as the dual version about the
hyperplane arrangement in $\mathfrak{n}^*$ of Konno's numerical ones
about the moment map value $\alpha$ in $\mathfrak{m}^*$. The most
important feature of ours is that they will enable us to analyze the
structure of toric hyperk{\"a}hler manifolds in a efficient way.

In section 4, after recalling the basic relationship between toric
varieties and toric hyperk{\"a}hler manifolds, namely the extended
core and the core of toric hyperk{\"a}hler manifold, and cotangent
bundles of toric varieties. Denoting the open sets identical to the
cotangent bundles of compact toric varieties in the core of toric
hyperk{\"a}hler manifold as $U_\epsilon, \epsilon \in \Theta_{cpt}$,
we prove the main theorem,

\begin{theorem}\label{thm:corecover}
Let $Y(\alpha,0)$ be a smooth toric hyperk{\"a}hler manifold with
nonempty core, then the canonical open set $U_\epsilon \cong
T^*X(\mathcal{A}_\epsilon), \epsilon \in \Theta_{cpt}$ is a covering
of $Y(\alpha,0)$, i.e. the cotangent bundles of compact toric
varieties in the core are enough to glue $Y(\alpha,0)$.
\end{theorem}

Finally, there are a Corollary asserting that each individual
$T^*X(\mathcal{A}_\epsilon)$ is dense in $Y(\alpha,0)$ and a
Conjecture that its complement is of complex dimension $n$ and
constituted by several toric varieties intersecting together.

\

\noindent \textbf{Acknowledgement:} The authors want to thank Prof.
Bin Xu, Prof. Bailin Song, Dr. Yihuang Shen and Dr. Yalong Shi for
valuable conversations, and special thank goes to Prof. Xiuxiong
Chen for the encouragement.

\section{Preliminary}
\subsection{Toric variety}
We first state the symplectic definition about toric varieties. The
real torus $T^d=\{(\zeta_1, \zeta_2, \cdots, \zeta_d) \in
\mathbb{C}^d, |\zeta_i|=1\}$ acts on $\mathbb{C}^d$ freely. Denote
$M$ the $m$-dimensional connected subtorus of $T^d$ whose Lie
algebra $\mathfrak{m} \subset \mathfrak{t}^d$ is generated by
integer vectors(which we shall always take to be primitive), then we
have the following exact sequence
\begin{equation*}
0\rightarrow \mathfrak{m}\xrightarrow{\iota}
\mathfrak{t}^d\xrightarrow{\pi} \mathfrak{n}\rightarrow 0,
\end{equation*}
\begin{equation*}\label{eq:exact:dual}
0\leftarrow {\mathfrak{m}}^*\xleftarrow{\iota^*}
(\mathfrak{t}^d)^*\xleftarrow{\pi^*} {\mathfrak{n}}^*\leftarrow 0,
\end{equation*}
where $\mathfrak{n}=\mathfrak{t}^d/\mathfrak{m}$ is the Lie algebra
of the $n$-dimensional quotient torus $N=T^d/M$ and $m+n=d$. For
simplicity, we omit the superscript $d$ over $\mathfrak{t}$ from now
on.

Let $\{e_i\}_{i=1}^d$ be the standard basis of $\mathfrak{t}$, then
$\pi(e_i)=u_i$ are also primitive. Denote $\{e^*_i\}_{i=1}^d$ the
dual basis of $\mathfrak{t}^*$ and $\{\theta_i\}_{i=1}^m$ some basis
span $\mathfrak{m}$. The action of $M$ on $\mathbb{C}^d$ admits a
moment map
\begin{equation*}
\mu(z)=\frac{1}{2}\sum_{i=1}^d|z_i|^2\iota^* e^*_i.
\end{equation*}
For any $\alpha \in \mathfrak{m}^*$, the symplectic quotient
$\mu^{-1}(\alpha)/M$ is a toric variety, denoted as $X(\alpha)$,
inheriting Kahler metric from $\mathbb{C}^d$ on it's smooth part(cf.
\cite{Gu94a}). The quotient torus $N$ has a residue circle action on
$X(\alpha)$ and gives rise to a moment map to $\mathfrak{n}^*$,
\begin{equation*}
\bar{\mu}([z])=\frac{1}{2}\sum_{i=1}^d|z_i|^2 e^*_i.
\end{equation*}
The image of this map is a convex polytope $\Delta$ called the
Delzant polytope of $X(\alpha)$(cf. \cite{De88}).

Conversely, any smooth compact toric variety $X$ of complex
dimension $n$, with a Kahler metric invariant under some torus $N$
comes from Delzant's construction. Unfortunately, this polytope does
not recover all the data of the quotient construction, and the worse
is that it does not cooperate well with the toric hyperk{\"a}hler
theory. We use the notion of hyperplanes arrangement with
orientation(cf. \cite{Pr04}) to replace polytope. In detail,
consider a set of rational oriented hyperplanes
$\mathcal{A}=\{(H_i,u_i)\}_{i=1}^d$,
\begin{equation*}
H_i=\{x \in \mathfrak{n}^*|\langle u_i,x\rangle+\lambda_i=0\},
\end{equation*}
where $H_i$ is the hyperplane and $u_i$ is fixed primitive vector in
$\mathfrak{n}_\mathbb{Z}$ specifying the orientation, called the
normal of $H_i$. We define several subspaces related to these
oriented hyperplanes,
\begin{equation*}
H^{\geq 0}_i=\{x \in \mathfrak{n}^*|\langle u_i,x \rangle+\lambda_i
\geq 0\},
\end{equation*}
\begin{equation*}
H^{\leq 0}_i=\{x \in \mathfrak{n}^*|\langle u_i,x \rangle+\lambda_i
\leq 0\}.
\end{equation*}
A ploytope is naturally associated to this arrangement,
\begin{equation*}
\Delta=\bigcap^d_{i=1} H^{\geq 0}_i,
\end{equation*}
which could be empty or unbounded.

Then the arrangement $\mathcal{A}$ will decide a toric variety the
same as $\Delta$ does. Since $u_i$ define a map $\pi: \mathfrak{t}
\rightarrow \mathfrak{n}$, where $\mathrm{Ker} \pi=\mathfrak{m}$,
let $M$ be the Lie group corresponding to $\mathfrak{m}$ and set
$\alpha=\sum \lambda_i \iota^* e_i^*$, then we call
$\mu^{-1}(\alpha)/M$ the toric variety corresponding to
$\mathcal{A}$ and $\lambda=(\lambda_1, \cdots, \lambda_d)$ a lift of
$\alpha$. For fixed normal vectors, the hyperplane arrangements
corresponding to two different lifts of same moment map value
$\alpha$ only differ by a parallel transport, thus produce same
toric variety(cf. \cite{Pr04}). So we can abuse the notations
$X(\alpha)$ and $X(\mathcal{A})$.

\begin{example}[see \cite{BD00} or \cite{Pr04}]\label{ex:2dim:arr}
Let $n=2$, $u_1=f_1$, $u_2=f_2$, $u_3=-f_1-f_2$, $u_4=-f_2$, and
$\lambda_1=\lambda_2=\lambda_3=\lambda_4=1$. The toric variety is
Hirzebruch surface $\mathbb{P}(\mathcal{O}\oplus\mathcal{O})$. See
Figure \ref{fig:2dim}.
\end{example}

\begin{figure}[h]
\centering \subfigure[the fan]{
\label{fig:2dim:fan} %% label for first subfigure
\includegraphics[width=2.2in]{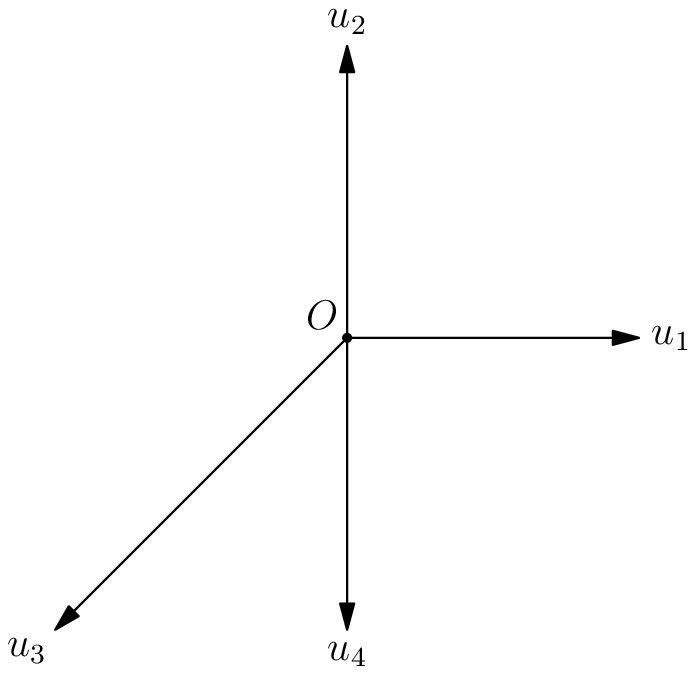}}
\hspace{0.3in} \subfigure[the hyperplane arrangement and polytope]{
\label{fig:2dim:arr} %% label for second subfigure
\includegraphics[width=2.2in]{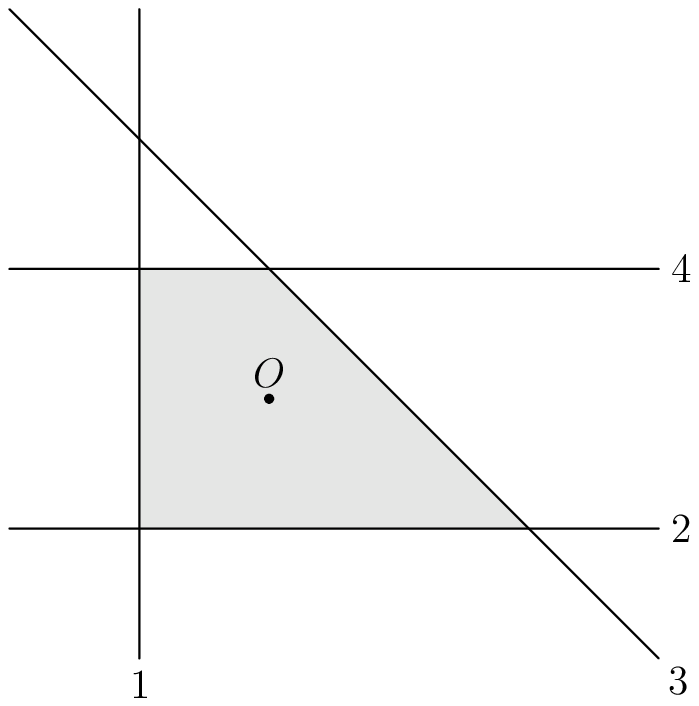}}
\caption{the fan and arrangement in Example \ref{ex:2dim:arr}}
\label{fig:2dim} %% label for entire figure
\end{figure}

If the rational vectors $u_i$ is regular, i.e. every collection of
$n$ linearly independent vectors $\{u_{i_1},\cdots, u_{i_n}\}$ span
$\mathfrak{n}_\mathbb{Z}$ as a $\mathbb{Z}$-basis, then
$\mathcal{A}$ is called regular. The arrangement $\mathcal{A}$ is
called simple if every subset of $k$ hyperplanes with nonempty
intersection intersects in codimension $k$. Then $\mathcal{A}$ is
smooth if it is both regular and simple. It is not difficult to see
that $X(\mathcal{A})$ is smooth if and only if $\mathcal{A}$ is
smooth.

From now on, $\mathcal{A}$ is always assumed to be smooth. If we
denote the regular value set of moment map $\mu$ as
$\mathfrak{m}_{reg}^*$, it is easy to check this condition is
equivalent to $\{u_i\}$ is regular and $\alpha \in
\mathfrak{m}_{reg}^*$. Moreover, denote $\Theta$ the set of maps
form $\{1, \dots, d\}$ to $\{-1,1\}$. For $\epsilon \in \Theta$, let
$\mathcal{A}_\epsilon$ be the arrangement changing the normal of
$H_i$ if $\epsilon(i)=-1$, and when $\epsilon(i)=1$ for all $i$, we
abbreviate the subscript $\epsilon$ for simplicity. Notice that the
toric variety $X(\mathcal{A}_\epsilon)$ for various $\epsilon$ could
be totally different.

\subsection{Toric hyperk{\"a}hler manifold}

A $4n$-dimensional manifold is hyperk{\"a}hler if it possesses a
Riemannian metric $g$ which is Kahler with respect to three complex
structures $I_1$; $I_2$; $I_3$ satisfying the quaternionic relations
$I_1 I_2 = -I_2 I_1 = I_3$ etc. To date the most powerful technique
for constructing such manifolds is the hyperk{\"a}hler quotient
method of Hitchin, Karlhede, Lindstrom and Rocek(\cite{HKLR87}). We
specialized on the class of hyperk{\"a}hler quotients of flat
quaternionic space $\mathbb{H}^d$ by subtori of $T^d$. The geometry
of these spaces turns out to be closely connected with the theory of
toric varieties.

Since $\mathbb{H}^d$ can be identified with $T^*\mathbb{C}^d \cong
\mathbb{C}^d\times\mathbb{C}^d$, it has three complex structures
$\{I_1, I_2, I_3\}$. The real torus $T^d=\{(\zeta_1, \zeta_2,
\cdots, \zeta_d) \in \mathbb{C}^d, |\zeta_i|=1\}$ acts on
$\mathbb{C}^d$ induce a action on $T^*\mathbb{C}^d$ keeping the
hyperk{\"a}hler structure,
\begin{equation*}
(z,w)\zeta=(z \zeta, w \zeta^{-1}).
\end{equation*}
The subtours $M$ acts on it admitting a hyperk{\"a}hler moment map
$\mu=(\mu_\mathbb{R}, \mu_\mathbb{C}): \mathbb{H}^d \rightarrow
\mathfrak{m}^* \times \mathfrak{m}^*_\mathbb{C}$, given by
\begin{equation*}\label{eq:moment:r}
\mu_\mathbb{R}(z,w) =
\frac{1}{2}\sum_{i=1}^d(|z_i|^2-|w_i|^2)\iota^* e^*_i,
\end{equation*}
\begin{equation*}\label{eq:moment:c}
\mu_\mathbb{C}(z,w) = \sum_{i=1}^d z_i w_i \iota^* e^*_i,
\end{equation*}
where the complex moment map $\mu_\mathbb{C} : \mathbb{H}^d
\rightarrow \mathfrak{m}^*_\mathbb{C}$ is holomorphic with respect
to $I_1$. Bielawski and Dancer introduced the definition of toric
hyperk{\"a}hler varieties, and generally speaking, they are not
toric varieties.

\begin{definition}[\cite{BD00}]
A toric hyperk{\"a}hler variety $Y(\alpha, \beta)$ is a
hyperk{\"a}hler quotient $\mu^{-1}(\alpha, \beta)/M$ for $(\alpha,
\beta) \in \mathfrak{m}^* \times \mathfrak{m}^*_\mathbb{C}$.
\end{definition}

A smooth part of $Y(\alpha,\beta)$ is a $4n$-dimensional
hyperk{\"a}hler manifold, whose hyperk{\"a}hler structure is denoted
by $(g, I_1, I_2, I_3)$. The quotient torus $N=T/M$ acts on
$Y(\alpha,\beta)$, preserving its hyperk{\"a}hler structure. This
residue circle action admits a hyperk{\"a}hler moment map
$\bar{\mu}=(\bar{\mu}_{\mathbb{R}}, \bar{\mu}_{\mathbb{C}})$,
\begin{equation*}
\bar{\mu}_\mathbb{R}([z,w]) =
\frac{1}{2}\sum_{i=1}^d(|z_i|^2-|w_i|^2) e^*_i,
\end{equation*}
\begin{equation*}
\bar{\mu}_\mathbb{C}([z,w]) = \sum_{i=1}^d z_i w_i  e^*_i.
\end{equation*}
Differs from the toric case, the map $\bar{\mu}$ to $\mathfrak{n}^*
\times \mathfrak{n}^*_\mathbb{C}$ is surjective, never with a
bounded image.

For the purpose of this article, we always assume that $Y(\alpha,
\beta)$ is a smooth manifold, i.e. $\{u_i\}$ is regular and
$(\alpha, \beta)$ is regular value of the moment map $\mu$(cf.
\cite{BD00} and \cite{Ko08}). Parallel with previous subsection, we
use hyperplane arrangement encoding the quotient construction. For
the moment map takes value in $\mathfrak{m}^* \times
\mathfrak{m}^*_\mathbb{C}$, the lift of $(\alpha,\beta)$ is
$\Lambda=(\lambda^1, \lambda^2, \lambda^3)$, s.t.
\begin{equation*}\label{eq:hyp:lif}
\begin{cases}
&\alpha=\sum \lambda_i^1 \iota^* e_i^*\\
&\beta=\sum (\lambda_i^2+\sqrt{-1}\lambda_i^3 )\iota^* e_i^*
\end{cases}
\end{equation*}
Then we can construct the arrangement of codimension 3 flats (affine
subspaces) in $\mathbb{R}^{3n}$,
\begin{equation*}
H_i=H_i^1 \times H_i^2 \times H_i^3,
\end{equation*}
where
\begin{equation}\label{eq:hyp:arr}
H_i^h=\{x \in \mathfrak{m}^*|\langle u_i,x \rangle+\lambda_i^h=0\},
\ (h=1,2,3, \ i=1,\dots,d)
\end{equation}
a prior with orientation $u_i$. For simplicity, we still denote it
as $\mathcal{A}$. Vice versa, such a arrangement of 3-flats
$\mathcal{A}$ determines a hyperk{\"a}hler quotient $Y(\alpha,
\beta)$.

We now investigate how the hyperk{\"a}hler quotient $Y(\alpha,
\beta)$ changes when the orientations are reversed. The original
subtorus $M$ is defined by the embedding $\iota \theta_k=a_k^i e_i$
and $(\alpha, \beta)$ has a lift $\Lambda$. Reversing the
orientation by letting $\tilde{u}_j=\epsilon(j) u_j$, we get
arrangement $\mathcal{A}_\epsilon$. The new subtorus $\tilde{M}$ is
defined by the embedding $\tilde{\iota} \theta_k=\tilde{a}_k^i e_i$
where $\tilde{a}_k^j=\epsilon(j)a_k^j$ for $k=1,\dots,m$, or
equivalent saying $\tilde{\iota}^* e_j^*=\epsilon(j)\iota^* e_j^*$.
This has a few consequences. One is that the new subtorus
$\tilde{M}$ acts on $\mathbb{H}^d$ is the same as the original
subtorus $M$ does on $\mathbb{H}^d$ by exchanging the positions of
$z_j$ and $w_j$ if $\epsilon(j)=-1$. The second is about the new
lift $\tilde{\Lambda}$. By Equation (\ref{eq:hyp:arr}), we know that
$\tilde{\Lambda}_j=\epsilon(j)\Lambda_j$, thus $\tilde{\alpha}=\sum
\tilde{\lambda}_i^1 \tilde{\iota}^* e_i^*=\sum \lambda_i^1 \iota^*
e_i^*=\alpha$, similarly $\tilde{\beta}=\beta$. Finally, the level
set becomes
\begin{equation*}
\frac{1}{2}\sum_{\epsilon(i)=1} (|z_i|^2-|w_i|^2)\iota^* e^*_i
+\frac{1}{2}\sum_{\epsilon(j)=-1} (|w_j|^2-|-z_j|^2)\iota^*
e^*_j=\alpha,
\end{equation*}

\begin{equation*}
\sum_{\epsilon(i)=1} z_i w_i \iota^* e^*_i+ \sum_{\epsilon(j)=-1}
w_j (-z_j) \iota^*e^*_j=\beta.
\end{equation*}
Mapping $(z_j,w_j)$ to $(w_j,-z_j)$, the level set is identical to
the original one, so is the hyperk{\"a}hler quotient. This means
that toric hyperk{\"a}hler manifolds according to the same
arrangement with different orientations will be biholomorphic to
each other, which is a significant difference from the toric
variety(see a similar discussion in \cite{Ko02}, Lemma 4.2). Even
though, to fix the position of variables, we still presume every
arrangement is with given normal.

\begin{example}[see \cite{BD00}]\label{ex:1dim:arr}
Let $\beta=0$ and take $\alpha$ defined by the arrangement
$u_1=-f_1$, $u_2=u_3=f_1$ in $\mathfrak{n}^1$ where $f_1$ is the
standard basis, and $\lambda_1=1$, $\lambda_2=-\frac{1}{2}$,
$\lambda_3=0$. The resulted toric hyperk{\"a}hler manifold
$Y(\alpha,0)$ is the desingulariztion of
$\mathbb{C}^2/\mathbb{Z}_3$(cf. \cite{HS02}, section 10).

\begin{figure}[h]
\centering \subfigure[the fan, where $u_2$, $u_3$ superposition]{
\label{fig:1dim:fan} %% label for first subfigure
\includegraphics[width=2.2in]{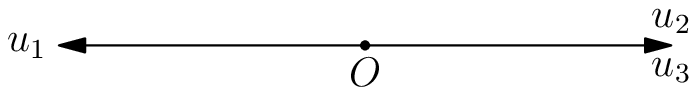}}
\hspace{0.3in} \subfigure[the hyperplane arrangement ]{
\label{fig:1dim:arr} %% label for second subfigure
\includegraphics[width=2.2in]{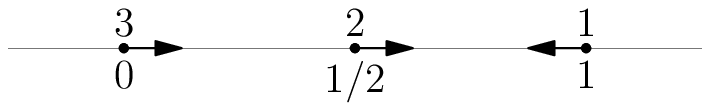}}
\caption{the fan and hyperplanes arrangement in Example
\ref{ex:1dim:arr}}
\label{fig:1dim} %% label for entire figure
\end{figure}

\end{example}

\section{GIT description and criteria for semi-stability}

\subsection{toric varieties} Let us consider the GIT
quotient of $\mathbb{C}^d$ by $M_\mathbb{C}$ with respect to the
linearization induced by $\alpha \in \mathfrak{m}^*_\mathbb{Z}$.
More explicitly, the element $\alpha$ induces the character $
\chi_\alpha: M_\mathbb{C} \rightarrow \mathbb{C}^\times$, where
$M_\mathbb{C}$ is the complexfication of $M$. Let $L^{\otimes m} =
\mathbb{C}^d \times \mathbb{C}$ be the trivial holomorphic line
bundle on which $M_\mathbb{C}$ acts by
\begin{equation*}
((z), v)_m \zeta = ((z) \zeta, v \chi_\alpha(\zeta)^m)_m.
\end{equation*}
A point $(z)$ is $\alpha$-semi-stable if and only if there exists
$m\in \mathbb{Z}_{>0}$ and a polynomial $f(p)$, where $p \in
\mathbb{C}^d$, such that $f((p)\zeta)= f(p)\chi_\alpha(\zeta)^m$ for
any $\zeta \in M_\mathbb{C}$ and $f(z) \neq 0$. We denote the set of
$\alpha$-semi-stable points in $\mathbb{C}^d$ by
$(\mathbb{C}^d)^{\alpha-ss}$, then there is a categorical quotient
$\phi: (\mathbb{C}^d)^{\alpha-ss} \rightarrow
(\mathbb{C}^d)^{\alpha-ss}//M_{\mathbb{C}}$, where
$(\mathbb{C}^d)^{\alpha-ss}//M_{\mathbb{C}}$ is the GIT quotient of
$\mathbb{C}^d$ by $M_\mathbb{C}$ respect to $\alpha$, more
precisely, the union of closed $M_{\mathbb{C}}$-orbits in
$(\mathbb{C}^d)^{\alpha-ss}$(cf. \cite{MFK94}, and the readers are
highly recommended to consult the lecture notes \cite{Do03} or
\cite{Th06} if they prefer the variety rather than the abstract
scheme setup). Sometimes $\mathbb{C}^d//_\alpha M_{\mathbb{C}}$
stands for this GIT quotient.

Unfortunately, the definition of stability respect to linearization
is only effective when $\alpha \in \mathfrak{m}^*_\mathbb{Z}$, i.e.
only corresponds to the algebraic toric variety with line bundle
described by Newton ploytope with integer vertices. Following
Konno's method in hyperk{\"a}hler case, it can be generalized to any
complex manifold.

\begin{lemma}\label{lem:sta:num}
Suppose that $\alpha \in \mathfrak{m}^*$,

\noindent (1) A point $z \in \mathbb{C}^d$ is $\alpha$-semi-stable
if and only if
\begin{equation}\label{eq:sta:num}
\alpha \in \sum^d_{i=1} \mathbb{R}_{\geq 0}|z_i|^2 \iota^* e_i^*.
\end{equation}

\noindent (2) Suppose $z \in (\mathbb{C}^d)^{\alpha-ss}$. Then the
$M_\mathbb{C}$-orbit through $z$ is closed in
$(\mathbb{C}^d)^{\alpha-ss}$ if and only if
\begin{equation*}\label{eq:clo:num}
\alpha \in \sum^d_{i=1} \mathbb{R}_{> 0}|z_i|^2 \iota^* e_i^*.
\end{equation*}
\end{lemma}
\begin{proof}
For convenience we give the proof of (1), and readers can find the
essential proof of (2) in \cite{Ko08}. Suppose $(z) \in
(\mathbb{C}^d)^{\alpha-ss}$. Then there exists $m \in
\mathbb{Z}_{>0}$ and a polynomial $f(p_1, \dots, p_d)$ such that
$f((p)\zeta) = f(p) \chi_\alpha(\zeta)^m$ and $f(z)\neq 0$. So we
can select out a monomial $f_0(p) = \prod_{i=1}^d p_i^{a_i}$, where
$a_i \in \mathbb{Z}_{>0}$, such that
$f_0((p)\zeta)=f_0(p)\chi_\alpha(\zeta)^m$ and $f_0(z)\neq0$. The
second condition implies that $a_i = 0 \ \text{if} \ z_i = 0$.
Moreover, the first condition implies $m \alpha = \sum_{i=1}^d a_i
\iota^* e_i^*$. To see this, let $\theta_k$ be the standard basis of
$\mathfrak{m}$ and $\rho \in \mathbb{C}^\times$, we have
$\chi(\mathrm{Exp}(\rho \theta_k))=e^{\rho\langle \alpha, \theta_k
\rangle}$ and $(p)\mathrm{Exp}(\rho \theta_k)=(p_i
e^{\rho\langle\iota\theta_k,e_i^*\rangle})=(p_i e^{\rho\langle
\iota^* e_i^*, \theta_k\rangle})$. Thus we proved Equation
(\ref{eq:sta:num}).
\end{proof}

This definition of stability coincides the original GIT one when
$\alpha \in \mathfrak{m}^*_\mathbb{Z}$. Followed by

\begin{proposition}

\noindent (1)If we fix $\alpha \in \mathfrak{m}^*$, then the natural
map $\sigma: X(\alpha) \rightarrow
(\mathbb{C}^d)^{\alpha-ss}//M_{\mathbb{C}}$ is a homeomorphism.

\noindent (2)If $\alpha \in \mathfrak{m}_{reg}^*$, then every
$M_\mathbb{C}$-orbit is closed in $(\mathbb{C}^d)^{\alpha-ss}$. So
the categorical quotient
$(\mathbb{C}^d)^{\alpha-ss}//M_{\mathbb{C}}$ is a geometric quotient
$(\mathbb{C}^d)^{\alpha-ss}/M_{\mathbb{C}}$.
\end{proposition}

Readers could consult \cite{Ko08} for the proof. Thus we can
identify the symplectic quotient $X(\alpha)$ with the GIT quotient
$(\mathbb{C}^d)^{\alpha-ss}//M_{\mathbb{C}}$ in both algebraic and
holomorphic case. This principle was established in \cite{KN78},
\cite{MFK94} in the algebraic case. The general holomorphic version
is proved in \cite{Na99}.

In practice, a geometric interpretation of the stability is needed.
Fix $\alpha$, it is equivalent to give an arrangement $\mathcal{A}$.
For any $z \in \mathbb{C}^d$, we associate a subregion of
$\mathfrak{m}_+^*$. Let
\begin{equation*}F_z(i)=
\begin{cases}
H_i^{\geq 0} & if z_i \neq 0\\
H_i & if z_i = 0
\end{cases},
\end{equation*}
where $H_i$ and it's orientation come naturally from $\mathcal{A}$,
denote $St_{\mathcal{A}}(z)=\bigcap_{i=1}^d F_z(i)$, which is a
polytope or subpolytope. Via different lifts, $St_{\mathcal{A}}(z)$
only differs a parallel transport. And we leave the proof of
Proposition \ref{pro:tor:sta} as a special case of toric
hyperk{\"a}hler manifold in next subsection, but let us add some
remarks about it. Let $J$ be a subset of $\{1,\cdots,d\}$, it is
easy to see, a point $\{z|z_i=0, \text{if} \ i \in J \ \text{and} \
z_i \neq 0 \ \text{if} \ i \in \bar{J}\}$ is semi-stable if and only
if the hyperplanes $\{H_i|i\in J\}$ have nonempty intersection
within the polytope $\Delta$ defined by $\mathcal{A}$. So our
interpretation of stability can be viewed as a replacement of
definition about $\mathbb{C}^d_{\Delta}$ of \cite{Gu94a},
\cite{BD00} in the hyperplane arrangement context, and has a natural
extension to the toric hyperk{\"a}hler case.

\begin{example}
Consider the hyperplanes arrangement of Fig \ref{fig:2dim:arr}, the
$St_\mathcal{A}(z)$ for $z_2=0$ is the bottom of the trapezoid, and
empty for $z_1=z_3=0$, thus the first kind point is semi-stable and
the second one is not semi-stable.
\end{example}

\subsection{Toric hyperk{\"a}hler manifold}

Similarly, consider the GIT quotient of
$\mu_{\mathbb{C}}^{-1}(\beta)$ by $M_\mathbb{C}$ with respect to the
linearization on the trivial line bundle induced by $\alpha \in
\mathfrak{m}^*_\mathbb{Z}$. Denote the set of $\alpha$-semi-stable
points in $\mu_{\mathbb{C}}^{-1}(\beta)$ by
$\mu_{\mathbb{C}}^{-1}(\beta)^{\alpha-ss}$, then there is a
categorical quotient $\phi: \mu_{\mathbb{C}}^{-1}(\beta)^{\alpha-ss}
\rightarrow
\mu_{\mathbb{C}}^{-1}(\beta)^{\alpha-ss}//M_{\mathbb{C}}$, where
$\mu_{\mathbb{C}}^{-1}(\beta)^{\alpha-ss}//M_{\mathbb{C}}$ contains
all closed $M_{\mathbb{C}}$ orbits in
$\mu_{\mathbb{C}}^{-1}(\beta)^{\alpha-ss}$. Parallel with toric
case, the stability condition can be generalized to any $\alpha \in
\mathfrak{m}^*$(cf. \cite{Ko08}).

\begin{lemma}
Suppose that $\alpha \in \mathfrak{m}^*$,

\noindent (1) A point $(z,w) \in \mu_{\mathbb{C}}^{-1}(\beta)$ is
$\alpha$-semi-stable if and only if
\begin{equation}\label{eq:hyp:sta}
\alpha \in \sum^d_{i=1} \mathbb{R}_{\geq 0}|z_i|^2 \iota^* e_i^* +
\sum^d_{i=1}\mathbb{R}_{\geq 0}|w_i|^2 (-\iota^* e_i^*).
\end{equation}

\noindent (2) Suppose $(z,w) \in
\mu_{\mathbb{C}}^{-1}(\beta)^{\alpha-ss}$. Then the
$M_\mathbb{C}$-orbit through $(z,w)$ is closed in
$\mu_{\mathbb{C}}^{-1}(\beta)^{\alpha-ss}$ if and only if
\begin{equation*}\label{eq:hyp:clo}
\alpha \in \sum^d_{i=1} \mathbb{R}_{> 0}|z_i|^2 \iota^* e_i^* +
\sum^d_{i=1}\mathbb{R}_{> 0}|w_i|^2 (-\iota^* e_i^*).
\end{equation*}
\end{lemma}

Thus we can identify the symplectic quotient $Y(\alpha, \beta)$ with
the GIT quotient
$\mu_{\mathbb{C}}^{-1}(\beta)^{\alpha-ss}//M_{\mathbb{C}}$ for any
$\alpha \in \mathfrak{m}^*$, and the geometric quotient
$\mu_{\mathbb{C}}^{-1}(\beta)^{\alpha-ss}/M_{\mathbb{C}}$ if
$(\alpha, \beta)$ is a regular value.

It time to examine the numerical stability condition Equation
(\ref{eq:hyp:sta}) a little further. For $\alpha$ lies in
$\mathfrak{m}^*$, there is
\begin{equation*}
\alpha=\sum_{i=1}^d x_i \iota^* e_i^*.
\end{equation*}
Assume $\iota^* e_i^*=a_i^k \theta_k^*$, and $\alpha=\alpha^k
\theta_k^*$, above equation turns to a linear equation system
\begin{equation*}
A x=\alpha,
\end{equation*}
where $A$ is $m \times d$ matrix with entry $a_i^k$, and $\alpha$
represents the column vector $\{\alpha^k\}_{k=1}^m$. As we know
$\alpha$ has lift $\lambda$ s.t. $\alpha=\sum_{i=1}^d \lambda_i
\iota^* e_i^*$ as a particular solution, so it is enough to consider
the undetermined homogeneous system
\begin{equation*}
A x=0.
\end{equation*}
If its $n$-dimensional solution space is denoted as $\mathfrak{N}$,
then the solution of original inhomogeneous system is
$\mathfrak{N}_\alpha \triangleq\lambda+\mathfrak{N}$, a $n$-plane in
$\mathbb{R}^d$.

\begin{proposition}\label{pro:arrcut}
Regarding the point $\lambda$ in $\mathfrak{N}_\alpha$ as the origin
, projecting of $\mathfrak{N}_\alpha$ onto some standard
$\mathbb{R}^n$, then we can identify
$\pi_{\mathbb{R}^n}(\mathfrak{N}_\alpha)$ with $\mathfrak{n}^*$, and
the hyperplane arrangement $H_i$ is defined by
$\pi_{\mathbb{R}^n}(\mathfrak{N}_\alpha \bigcap \{x_i=0\})$, where
$\{x_i=0\}$ is the coordinate hyperplane in $\mathbb{R}^d$.
\end{proposition}

\begin{proof}
Denote the standard basis of $\mathbb{R}^d$ as $\partial_i$. If $x
\in \mathfrak{N}_\alpha \bigcap\{x_i=0\}$ is a vector in
$\mathbb{R}^d$, by definition, $ \langle x,
\partial_i \rangle=0$. On another hand, $x$ corresponds the vector $x-\lambda$ in
$\mathfrak{N}_\alpha$, thus $\langle x-\lambda,
\partial_i \rangle=\langle -\lambda, \partial_i \rangle=-\lambda_i$. Now $\mathfrak{N}_\alpha \bigcap
\{x_i=0\}$ looks a little different from the hyperplane arrangement
in $\mathfrak{n}^*$(see Fig \ref{fig:Arr}). We have to project it
onto some $\mathbb{R}^n$. If $\partial_j \in \mathbb{R}^n$, then we
simply have $u_j=\pi_{\mathbb{R}^n}
\partial_j$, otherwise, the normal $u_j$ will be the unique vector preserving
above inner product which can be calculated by a little Euclidean
geometry. Thus we recover the arrangement, and it is easy to see
that the resulted arrangement to different projections differ at
most $GL(n, \mathbb{Z})$ transforms, i.e. all equivalent.
\end{proof}

\begin{example}
Consider the 1-dimensional torus acts on $\mathbb{H}^3$
diagonally,and choose $\alpha=3$ with lift $\lambda=(1,1,1)$, we
have a two dimensional solution space and its projection onto
$\mathfrak{n}^* \cong \mathbb{R}^2$ illustrated in Fig
\ref{fig:Arr}, where $O$ is the projection of $\lambda$.
\end{example}

\begin{figure}[h]
\centering \subfigure[the solution space $\mathfrak{N}_\alpha$]{
\label{fig:Arr:Solu} %% label for first subfigure
\includegraphics[width=2.3in]{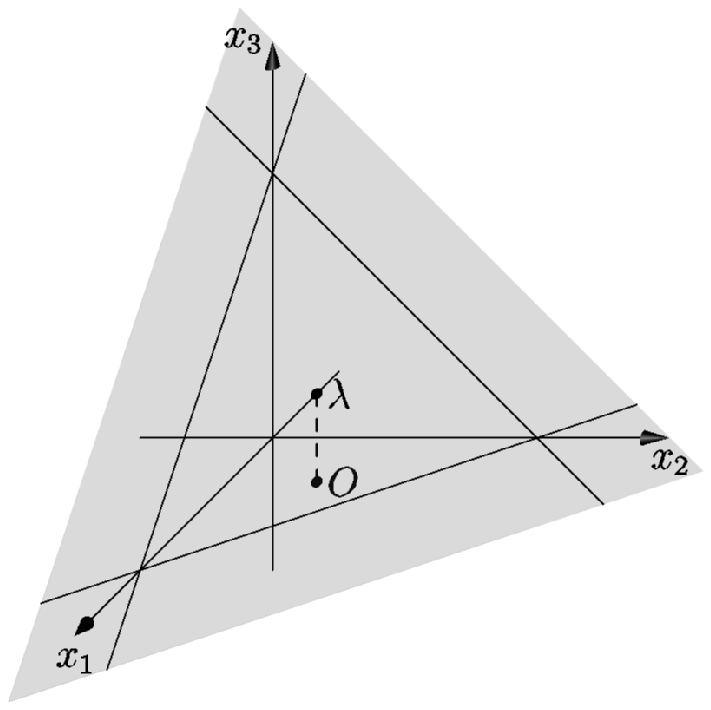}}
\hspace{0.3in} \subfigure[projected onto $x_1x_2$ plane identifies
with the arrangement in $\mathfrak{n}^*$]{
\label{fig:Arr:Proj} %% label for second subfigure
\includegraphics[width=2.1in]{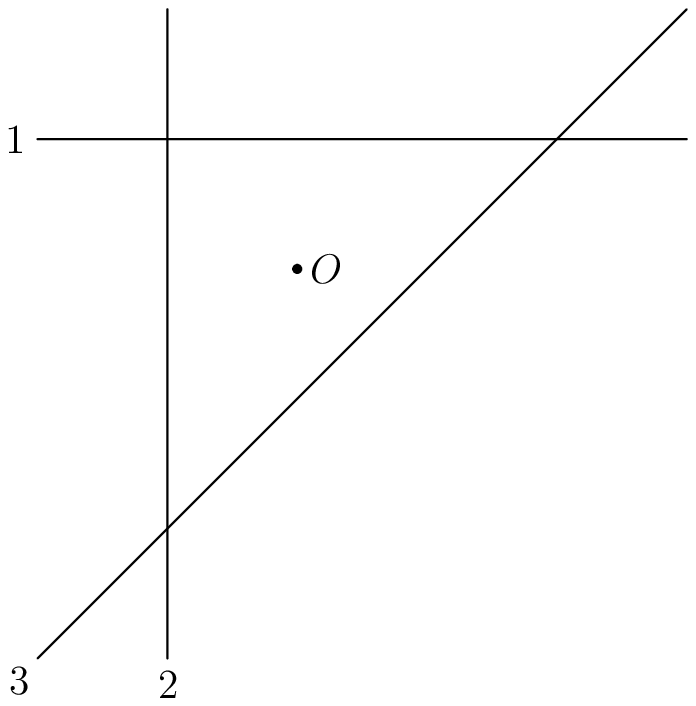}}
\caption{the relation of solution space and hyperplane arrangements}
\label{fig:Arr} %% label for entire figure
\end{figure}

Notice that the projection $\pi_{\mathbb{R}^n}$ does not affect the
intersection relations of the hyperplane in $\mathfrak{N}_\alpha$
and $\mathfrak{n}^*$, i.e. the relative positions. So we will not
distinguish them from each other in the later practice. In effect,
the $\mathfrak{N}_\alpha$ picture is more close to the intrinsic
geometry of toric variety and toric hyperk{\"a}hler manifold.

%\subsubsection{Geometric interpretation of stability}

As promised, we give the geometric interpretation of the stability
condition. Recall Equation (\ref{eq:hyp:sta}), we easily have

\begin{lemma}\label{lem:sta:alg}

Given $I$, $J$ subset of $\{1,\dots,d\}$, the point $\{(z,w)|z_i=0,
w_j=0, \text{if and only if} \ i \in I, j \in J\}$ is
$\alpha$-semi-stable if and only if there exists a solution $x \in
\mathfrak{N}_\alpha$ s.t. $x_i \leq 0, x_j \geq 0, for \ i \in I, j
\in J$.
\end{lemma}

Fixing the arrangement $\mathcal{A}$, we will associate a region to
every point $(z,w) \in \mathbb{H}^d$. Set
\begin{equation*}F_z(i)=
\begin{cases}
H_i^{\geq 0} & if \ z_i \neq 0\\
H_i & if \ z_i = 0
\end{cases},
\end{equation*}
and
\begin{equation*}F_w(i)=
\begin{cases}
H_i^{\leq 0}& if \ w_i \neq 0\\
H_i & if \ w_i = 0
\end{cases},
\end{equation*}
define $St_{\mathcal{A}}(z,w)=\bigcap_{i=1}^d (F_z(i)\bigcup
F_w(i))$, which is a union of polytopes or subpolytopes. Then we are
in the position to prove Proposition \ref{pro:hyp:sta}.

\begin{proof}
We already know the hyperplane arrangement is the solution space cut
by the coordinate hyperplanes in $\mathbb{R}^d$. The non empty of
$St_{\mathcal{A}}(z,w)$ of the point $\{(z,w)|z_i=0, w_j=0,\
\text{if and only if} \ i \in I, j \in J\}$ implies that there is a
$x \in \mathfrak{N}_\alpha$ s.t. $x_i \leq 0, x_j \geq 0, for \ i
\in I, j \in J$. By Lemma \ref{lem:sta:alg}, This means this point
is semi-stable, and vice versa.
\end{proof}

Proposition \ref{pro:tor:sta} concerning the toric case is an easy
consequence by letting all $w_i$ be zero.

\section{Coverings of toric hyperk{\"a}hler manifolds}
In this part, $\beta$ is taken to be zero. It is enough to merely
consider the hyperplanes arrangement $H^1$. We will abuse the
notation using $H$ in stead of $H^1$.

\subsection{The geometry of the extended core and core} The subset
of $Y(\alpha,0)$
\begin{equation*}
Z=\bar{\mu}_\mathbb{C}^{-1}(0)=\{[z,w] \in Y(\alpha,0)| z_i w_i=0 \
\text{for all} \ i\},
\end{equation*}
is called the extended core by Proudfoot(cf. \cite{Pr04}), which
naturally breaks into components
\begin{equation*}
Z_\epsilon=\{[z,w] \in Y(\alpha,0)| w_i=0 \ if \ \epsilon(i)=1 \ and
\ z_i=0 \ if \ \epsilon(i)=-1\}.
\end{equation*}
The variety $Z_\epsilon \subset Y(\alpha,0)$ is a $n$-dimensional
isotropy Kahler subvariety of $Y(\alpha,0)$ with an effective
hamiltonian $T^n$-action, hence a toric variety itself. It is not
hard to see this is just the toric variety corresponding to the
oriented hyperplane arrangement $\mathcal{A}_\epsilon$. Denote the
associated polytope of $\mathcal{A}_\epsilon$ as $\Delta_\epsilon$.
The set $Z_{cpt}=\bigcup_{\epsilon \in \Theta_{cpt}} Z_\epsilon$,
where $\Theta_{cpt}=\{\epsilon|\Delta_\epsilon \ \text{bounded}\}$,
is called the core, union of compact toric varieties
$X(\mathcal{A}_\epsilon), \epsilon \in \Theta_{cpt}$.

It is natural to ask when a toric hyperk{\"a}hler manifold has
nonempty core. Let $\mathcal{A}$ be the smooth arrangement and $u_i$
its normals, it amounts to check wether $\mathcal{A}$ contains
bounded $n$-dimensional polytopes. Then $\mathcal{A}$ does not
contain bounded polytopes if and only if there is a subset $K$ of
$\{1,\dots,d\}$, the length $|K|<n$, s.t. for each $k \in K$, $u_k$
is independent of all others $\{u_i|i \in \{1,\dots,d\}, i \neq
k\}$. In this case, $Y(\alpha,0)$ can be written as
$\tilde{Y}(\tilde{\alpha},0)\times \mathbb{H}^{|K|}$, where
$\tilde{Y}(\tilde{\alpha},0)$ is a $4(n-|K|)$ dimensional toric
hyperk{\"a}hler manifold(see Fig \ref{fig:core:empty}). Thus except
the trivial product case where some $u_i$ is linear independent of
all others, the toric hyperk{\"a}hler variety $Y(\alpha,0)$ has
nonempty core, which will be take for granted in the remaining part.

\begin{figure}[h]
\centering \subfigure[trivial product
$\widehat{\mathbb{C}^2/\mathbb{Z}_2} \times \mathbb{H}$ with empty
core]{
\label{fig:core:empty} %% label for first subfigure
\includegraphics[width=2.2in]{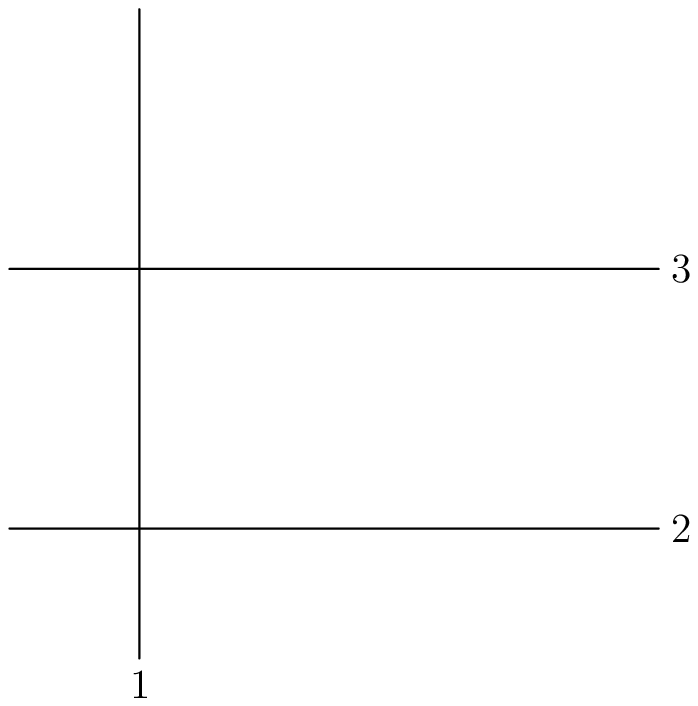}}
\hspace{0.3in} \subfigure[with nonempty core $\mathbb{C}P^2$
intersecting with Hirzebruch surface]{
\label{fig:core:nonempty} %% label for second subfigure
\includegraphics[width=2.2in]{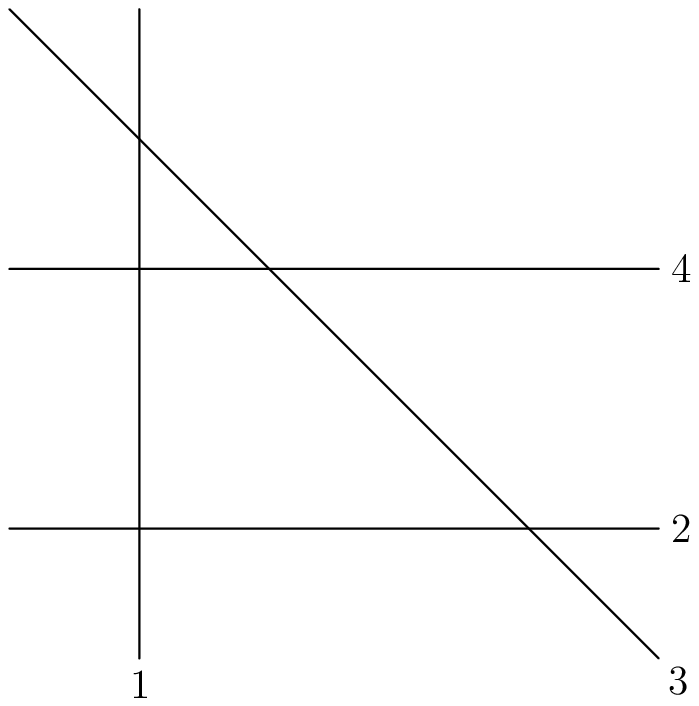}}
\caption{bounded polytopes in hyperplane arrangements}
\label{fig:core} %% label for entire figure
\end{figure}

\subsection{Cotangent bundle of toric variety} Another important
result in \cite{BD00} concerns the cotangent bundle of toric
variety. Suppose $\mathcal{A}$ is nonempty, the toric variety
$X(\mathcal{A})$ is automatically in the core of $Y(\alpha,0)$.
Dancer and Bielawski proved that $T^*X(\mathcal{A})$ isomorphic to
an open subset $U$. Later, this result was generalized by Konno(see
\cite{Ko02}, lemma 4.2), namely

\begin{lemma}\label{lem:cotantor}
Let $Y(\alpha,0)$ be a toric hyperk{\"a}hler manifold. If
$X(\mathcal{A}_\epsilon)$ is not empty, then its cotangent bundle
$T^*X(\mathcal{A}_\epsilon)$ is contained in $Y(\alpha,0)$ as an
open subset. Moreover, the hyperk{\"a}hler metric restricted on the
zero section of this cotangent bundle, is the canonical metric on
toric variety $X(\mathcal{A}_\epsilon)$.
\end{lemma}

From now on, the open set isomorphic to $T^*X(\mathcal{A}_\epsilon)$
is denoted as $U_\epsilon$(stands for $U_{\mathcal{A}_\epsilon}$ in
effect). We will only sketch the proof in case $\epsilon(i)=1$ for
all $i$, and reader could consult \cite{Ko02} for the detail.
Consider the open subset $(\mathbb{C}^d)^{\alpha-ss} \times
\mathbb{C}^d$ of $\mathbb{H}^d$. The group $M$ acts on it freely, so
we can perform the hyperk{\"a}hler quotient construction on
$(\mathbb{C}^d)^{\alpha-ss} \times \mathbb{C}^d$ and obtain an open
subset $U$ of $Y(\alpha,0)$. In order to show that $U$ is isomorphic
to $T^*X(\mathcal{A})$, we identify $U$ with the GIT quotient
$((\mathbb{C}^d)^{\alpha-ss} \times \mathbb{C}^d) \cap
\mu_{\mathbb{C}}^{-1}(0)/M_\mathbb{C}$, i.e.
\begin{equation*}
\{(z,w) \in (\mathbb{C}^d)^{\alpha-ss} \times
\mathbb{C}^d|\sum_{i=1}^d z_i w_i \iota^* e^*_i=0\}/M_\mathbb{C},
\end{equation*}
where we directly apply geometric quotient for $(\alpha,0)$ is a
regular value. Above equation simply says that the vector $w \in
T^*_z (\mathbb{C}^d)^{\alpha-ss}$ annihilates the tangent vectors
along the $M_\mathbb{C}$ orbits, i.e. the vertical tangent vectors
of the projection $(\mathbb{C}^d)^{\alpha-ss} \rightarrow
X(\mathcal{A})$. This show that $U$ is isomorphic to
$T^*X(\mathcal{A})$.

\subsection{The proof of the main theorem}

All the nonempty canonical open sets $U_{\epsilon}$ naturally
constitute a covering of $Y(\alpha,0)$. Bielawski pointed it out
that in \cite{Bi99}, the toric hyperk{\"a}hler manifold can be
viewed as the gluing the cotangent bundle of some toric varieties.
So it means that we can choose all the toric varieties in the
extended core to do this job, but in a subtle way. The only chance
to see how these cotangent bundles are glued is when there is as few
as components are involved, i.e. using few cotangent bundles to glue
the toric hyperk{\"a}hler manifold. In the following, we will see
compact toric varieties in the core play such a role.

\begin{theorem}
Let $Y(\alpha,0)$ be a smooth toric hyperk{\"a}hler manifold with
nonempty core, then the canonical open set $U_\epsilon \cong
T^*X(\mathcal{A}_\epsilon), \epsilon \in \Theta_{cpt}$ is a covering
of $Y(\alpha,0)$, i.e. the cotangent bundles of compact toric
varieties in the core are enough to glue $Y(\alpha,0)$.
\end{theorem}

\begin{proof}
We first check the stability condition with respect to the toric
variety $X(\mathcal{A}_\epsilon), \epsilon \in \Theta_{cpt}$. The
point $(z_J, w_{\bar{J}}) \in \mathbb{C}^d$, where
$J=\{i|\epsilon(i)=1\}$ and $\bar{J}$ is its complement, is
semi-stable if and only if
\begin{equation}\label{eq:sta:epsilon}
\alpha \in \sum_J \mathbb{R}_{\geq 0}|z_i|^2 \iota^* e_i^* +
\sum_{\bar{J}} \mathbb{R}_{\geq 0}|w_i|^2 (-\iota^* e_i^*).
\end{equation}
If we denote these points by $(\mathbb{C}^d)^{\alpha-ss}_\epsilon$,
then $(\mathbb{C}^d)^{\alpha-ss}_\epsilon \times \mathbb{C}^d$ is a
subset of $(\mathbb{H}^d)^{\alpha-ss}$. Performing the
hyperk{\"a}hler quotient on it results in $U_\epsilon$, which
isomorphic to $T^*X(\mathcal{A}_\epsilon)$. To prove the union of
$U_\epsilon$ covers $Y(\alpha,0)$, it suffices to show that the
union of $(\mathbb{C}^d)^{\alpha-ss}_\epsilon \times \mathbb{C}^d$
covers $(\mathbb{H}^d)^{\alpha-ss}$. In practice, we show that if a
point lies out side for every $(\mathbb{C}^d)^{\alpha-ss}_\epsilon
\times \mathbb{C}^d$, then it must be unstable.

We claim that if $(z,w)$ lies out side of
$(\mathbb{C}^d)^{\alpha-ss}_\epsilon \times \mathbb{C}^d$ for some
fixed $\epsilon \in \Theta_{cpt}$, then $St_{\mathcal{A}}(z,w) \cap
\Delta_\epsilon=\emptyset$. Without losing any generality, we assume
$\epsilon(i)=1$. The point $\{z|z_i=0, \text{if and only if} \ i \in
I\}$ is unstable for $X(\mathcal{A})$ means that $\{H_i\}_{i \in I}$
do not have intersection in $\Delta$. Since $z_i=0$ for $i \in I$,
the region $St_{\mathcal{A}}(z,w)$ now lies in the intersection of
half spaces $\{H^{0 \leq}_i\}_{i \in I}$, $St_{\mathcal{A}}(z,w)$
can not contact with $\Delta$. If not, let $x$ be the common
intersection of $\Delta$ and $St_{\mathcal{A}}(z,w)$, then $x \in
H^{\geq 0}_i$ and $x \in H^{\leq 0}_i$, for all the $i \in I$. It
means that $x \in \bigcap_{i \in I} H_i$ is a common intersection of
$H_i$, contradiction with $z$ is unstable with respect to the toric
variety.

Same argument applies for any $\epsilon \in \Theta_{cpt}$, so
$St_{\mathcal{A}}(z,w)$ are not adjacent to any $\Delta_\epsilon$.
While $\Delta_\epsilon$ are all the bounded polytopes and every
unbounded polytope and its subpolytopes must have intersection with
one of them(for example, see Fig \ref{fig:core:nonempty}), thus the
only possibility is $St_{\mathcal{A}}(z,w)$ is empty, i.e. $(z,w)$
is unstable.
\end{proof}

This result is not surprising. By the construction of the open set
$U_\epsilon \cong T^* X(\mathcal{A}_\epsilon)$ from GIT method, we
have in effect

\begin{corollary}\label{cor:dense}
Let $Y(\alpha,0)$ be a toric hyperk{\"a}hler manifold. If
$X(\mathcal{A}_\epsilon)$ is not empty, then its cotangent bundle
$T^*X(\mathcal{A}_\epsilon)$ is contained in $Y(\alpha,0)$ as an
open dense subset.
\end{corollary}

\begin{proof}
Check the semi-stable condition Equation (\ref{eq:sta:epsilon}),
there are only two possibilities. One is the set
$(\mathbb{C}^d)^{\alpha-ss}_\epsilon$ is empty. Otherwise, if there
is some point $(z_J,w_{\bar{J}})$ satisfying the condition, then any
$\{(z_J,w_{\bar{J}})|z_i \neq 0 \ \text{for} \ \epsilon(i)=1 \
\text{and} \ w_i \neq 0 \ \text{for} \ \epsilon(i)=-1\}$ belong to
$(\mathbb{C}^d)^{\alpha-ss}_\epsilon$, which means
$(\mathbb{C}^d)^{\alpha-ss}_\epsilon$ is dense in $\mathbb{C}^d$.
Then the corollary follows from that $\mu_\mathbb{C}^{-1}(0) \bigcap
(\mathbb{C}^d)^{\alpha-ss}_\epsilon \times \mathbb{C}^d$ is dense in
$\mu_{\mathbb{C}}^{-1}(0)^{\alpha-ss}$ and on which the complex
torus $M_\mathbb{C}$ acts effectively.
\end{proof}

Further more, we can verify that the following conjecture holds in
most low dimensional cases
\begin{conjecture}
The complement of each open dense set $U_\epsilon \cong T^*
X(\mathcal{A}_\epsilon)$ in the toric hyperk{\"a}hler manifold
$Y(\alpha,0)$ is of complex codimension $n$ and constituted by
several toric varieties in the extended core of $Y(\alpha,0)$.
\end{conjecture}

The method developed in the proof of the main theorem is not
effective in this problem probably because that the state set
$St_{\mathcal{A}}(z,w)$ is not adjacent to some $\Delta_\epsilon$ is
much weaker than $St_{\mathcal{A}}(z,w)$ is empty itself. So it is
hopeful utilizing deep combinatorial theory to prove this
conjecture.

\bibliographystyle{alpha}
\bibliography{hkBib}

\iffalse

\fi

\newpage

\noindent Craig van Coevering: craigvan@ustc.edu.cn

\

\noindent Wei Zhang: zhangw81@ustc.edu.cn

\

\

\noindent School of Mathematical Sciences

\noindent University of Science and Technology of China

\noindent Hefei, 230026, P. R.China

\end{document}